\documentclass[11pt]{article}
\usepackage{amssymb,amsmath}
\usepackage{amsthm}
\usepackage{amscd}
\newtheorem{theorem}{Theorem}[section]
\newtheorem{theorem*}{Theorem A\!\!}
\newtheorem{proposition}{Proposition}[section]
\newtheorem{proposition*}{Proposition A\!\!}
\newtheorem{corollary}{Corollary}[section]
\newtheorem{corollary*}{Corollary A\!\!}
\newtheorem{definition*}{Definition}

\newtheorem{lemma}{Lemma}[section]

\DeclareMathOperator{\Arg}{Arg}
\DeclareMathOperator{\diag}{diag}

\DeclareMathOperator{\Exp}{Exp}

\DeclareMathOperator{\Ad}{Ad}
\DeclareMathOperator{\ad}{ad}

\DeclareMathOperator{\id}{id}

\DeclareMathOperator{\im}{im}    
\DeclareMathOperator{\Lie}{Lie}

\begin{document}
\title{The symplectic area of a geodesic triangle in a  Hermitian symmetric space of compact type}

\author{Mads Aunskj\ae r Bech, Jean-Louis Clerc \& Bent \O rsted}

\date{ }

\maketitle
\begin{abstract} Let $M$ be an irreducible Hermitian symmetric space of compact type, and let $\omega$ be its K\"ahler form.  For a triplet $(p_1,p_2,p_3)$ of points in $M$ we study conditions under which a geodesic triangle $\mathcal T(p_1,p_2,p_3)$ with vertices $p_1,p_2,p_3$ can be unambiguously  defined. We consider the integral $A(p_1,p_2,p_3)=\int_\Sigma \omega$, where $\Sigma$ is a surface filling the triangle $\mathcal T(p_1,p_2,p_3)$ and discuss the dependence of $A(p_1,p_2,p_3)$ on the surface $\Sigma$. Under mild conditions on the three points, we prove an explicit formula for $A(p_1,p_2,p_3)$ analogous to the known formula  for the symplectic area of a geodesic triangle in a non-compact Hermitian symmetric space. 

\end{abstract}

{2010 Mathematics Subject Classification : 32M15, 53C55, 57T15\\Key words : compact Hermitian symmetric space, K\"ahler form, geodesic triangle, automorphy kernel}
\section*{Introduction}

On a K\"ahlerian manifold $M$, the existence of a Riemannian (Hermitian) metric and of a K\"ahler form $\omega$ suggests to define the \emph{symplectic area of a geodesic triangle}. Given three points $p_1,p_2,p_3\in M$, consider the geodesic triangle  $\mathcal T(p_1,p_2,p_3)$ obtained by joining $p_1, p_2$ (resp. $(p_2,p_3)$, $(p_3,p_1)$) by a geodesic segment. Choose a surface $\Sigma(p_1,p_2,p_3)$  in $M$ having the geodesic triangle as its (oriented) boundary and define 

\[A(p_1,p_2,p_3)=\int_{\Sigma(p_1,p_2,p_3)} \omega\ 
\]
as the \emph{symplectic area of the geodesic triangle} built on the vertices $(p_1,p_2,p_3)$. Loosely speaking, as $\omega$ is a closed form, a continuous variation of the surface does not change the value of the integral and so the formula defines a real-valued $3$-points function on $M$, which is in particular invariant under any holomorphic isometry of $M$.

However, there are two main obstacles to a rigorous definition of the symplectic area of a geodesic triangle : some coming from the \emph{Riemannian geometry} of $M$, some coming from the \emph{topology} of $M$.

In fact, to build the geodesic  triangle in a uniquely defined way, there should exists a \emph{unique minimizing geodesic segment} between any two points of $M$. Whether this is true for ``small'' triangles, it is in general not true globally, due to the existence of a \emph{cut locus} on $M$. Next,  when the manifold $M$ has a non-trivial topology, the integral may depend on the choice of the surface filling the triangle.

When $M$ is a Hermitian symmetric space of the non compact type, there exists a unique geodesic segment between two arbitrary points of $M$ (a consequence of the negative curvature) and the  topology of $M$ is trivial, as $M$ can be realized as the open ball of a complex vector space $\mathbb C^N$ for some Banach norm (a consequence of the \emph{ Harish Chandra embedding}). Hence the symplectic area of a geodesic triangle is well defined. An explicit expression for this 	area was obtained by Domic and Toledo (see \cite{dt}) for classical domains and in general for all domains by the two last present authors (see \cite{co2}) which amounts to
\begin{equation}
\int_{\Sigma(z_1,z_2,z_3)} \omega = -\big(\arg k(z_1,z_2)+\arg k(z_2,z_3)+\arg k(z_3,z_1)\big)\ ,
\end{equation}
 where $k(z,w)$ is a normalized version of the  \emph{Bergman kernel} of $M$ in its Harish Chandra realization.The three points function thus obtained defines a bounded cocycle, invariant under the group of holomorphic isometries of $M$, which turned out to be quite useful for various geometric problems (see e.g. \cite{dt, co2, w}).

The goal of this article is to clarify the definition of the symplectic area of a geodesic triangle  for a \emph{Hermitian symmetric space of compact type } and to find an analogue of the Domic-Toledo formula. For previous  work on this question, see \cite{b1, b2, bs, hm}. 

The geometry forces to restrict the definition to regular triangles, avoiding the cut-locus phenomena, and the topology suggests that one should consider rather the quantity
\[e^{\frac{i}{2} A(p_1,p_2,p_3)}
\]
as it can be shown to be independent of the surface chosen to fill the triangle,
see Theorem 4.2 for the precise formulation. The final result is a (singular) three-points function, which is a (multiplicative $U(1)$-valued) cocycle invariant under holomorphic isometries of the manifold. An appropriate version of the Domic-Toledo formula is then formulated and proved along lines similar to the original proof.
See Theorem 6.2 for the main formula.

Some results in this paper first appeared  in the thesis (\cite{mab}) defended by the first author at Aarhus University. The second author would like to thank Aarhus University for welcoming him during the time this paper was started.

{\it Acknowledgement:} Joseph Wolf has many contributions to Lie theory and geometry,
in particular that of Hermitian symmetric spaces; it is a
pleasure to let this paper be part of a tribute to him.

\section{The geometry of Hermitian symmetric spaces of compact type}

Any Hermitian symmetric space $M$ of compact type belongs to the class of \emph{symmetric $R$-spaces}, which means in particular that one can enlarge the group $U$ of holomorphic isometries of the space to a larger finite dimensional Lie group of diffeomorphisms of $M$. In the present case, this larger group turns out be be a \emph{complexification} $\mathbf G$ of $U$. So it is convenient to introduce the corresponding complex data, also  relevant for the dual  (non compact) Hermitian symmetric space $M^d$. We consider only \emph{irreducible} Hermitian symmetric spaces, which  correspond to \emph{ simple} Lie groups $G$ (or equivalently simple Lie algebras) of Hermitian type. Our main references for this section are \cite{h, s} and \cite{kor}.

We first introduce the infinitesimal data. So let $\mathfrak g_0$ be a simple real Lie algebra, choose a Cartan involution $\theta$ of $\mathfrak g_0$ and let $\mathfrak g_0 = \mathfrak k_0 \oplus \mathfrak p_0$ be the corresponding Cartan decomposition. The algebra is said to be \emph{of Hermitian type} if the center $\mathfrak z$ of $\mathfrak k_0$ is non trivial, in which case it can be shown to be one-dimensional. Then, up to a sign $\pm$, there exists a unique element $H_0$ in $\mathfrak z$ such that $(\ad_{\mathfrak p_0}H_0)^2 = -\id_{\mathfrak p_0}$. Then $J= \ad_{\mathfrak p_0}H_0$ defines a complex structure on $\mathfrak p_0$. Let $\mathfrak g, \mathfrak k, \mathfrak p$ be the complexifications of, respectively $ \mathfrak g_0, \mathfrak k_0, \mathfrak p_0$. 

Let $\mathfrak u = \mathfrak k_0\oplus i\mathfrak p_0$, and let $\tau$ be the conjugation of $\mathfrak g$ with respect to $\mathfrak u$. Then $\mathfrak u$ is a simple real Lie algebra  of compact type, and the pair $(\mathfrak u, \theta)$ where, abusing somewhat notation, $\theta$ is used for the restriction to $\mathfrak u$ of the complexification of $\theta$ is a simple \emph{symmetric Lie algebra of compact type}.

Let $\mathbf G$ be the simply connected Lie group with Lie algebra $\Lie(\mathbf G) = \mathfrak g$. The involution $\tau$ can be lifted to an involution of $\mathbf G$ (still denoted by $\tau$) and the fixed points set of $\tau$ is a maximal compact subgroup of  $U=\mathbf G^\tau$ of $\mathbf G$ with Lie algebra $\mathfrak u$. The complex automorphism $\theta$ can be lifted to an involution (still denoted by $\theta$) of $\mathbf G$ which preserves $U$. The set of fixed points of $\theta_{\vert U}$ is a subgroup $K_0$, with $\Lie(K_0) = \mathfrak k_0$. Form the quotient $M =U/K_0$.We refer to the element  $o=eK_0$ as the \emph{origin} of $M$. The tangent space $T_oM$ to $M$ at $o$ is identified with $i\mathfrak p_0$. The restriction of $\ad H_0$ to $i\mathfrak p_0$ induces a complex structure on $T_oM$. Let $B$ be the (complex)  Killing form of $\mathfrak g$. Then $q(X,Y)=-\frac{1}{2} B(X,\tau Y)$ is a positive-definite Hermitian form on $\mathfrak g$ and induces an inner product on $i \mathfrak p_0 \simeq T_o M$ which is invariant under the adjoint action of $K_0$. This allows to define a Hermitian metric on $M$, which realizes $M$ as a \emph{Hermitian symmetric space of compact type}.

As mentioned in the beginning of this section, the group $\mathbf G$ also acts on $M$. The involutions $\theta$ and $\tau$ lift to involutions of the group $\mathbf G$, still denoted by $\theta$ and $\tau$. Let $\mathfrak p_{\pm}$ be the eigenspaces of $\ad_{\mathfrak p} H_0$. Then $[ \mathfrak k, \mathfrak p_\pm]\subset \mathfrak p_\pm$ and $\mathfrak p^\pm$ turn out to be Abelian subspaces of $\mathfrak g$. Let $\mathbf P_\pm = \exp \mathfrak p_\pm$  be the corresponding Lie subgroups of $\mathbf G$, and let $\mathbf K= \mathbf G^\theta$ be the fixed points subgroup of $\theta$ in $\mathbf G$ with $Lie(\mathbf K) = \mathfrak k$. Then $\mathbf K \mathbf P_\pm$ are maximal parabolic subgroups of $\mathbf G$ and $ U\cap \mathbf K\mathbf P_- = U\cap K_0$, so that 
\[\mathbf G/\mathbf K\mathbf P_-\simeq U/K_0=M\ .
\]
The map 
$\Xi : \mathfrak p _+\ \longrightarrow M $ defined for $z\in \mathfrak p_+$ by
\[\Xi(z) = \exp(z) \mathbf K\mathbf P_-
\]
is a holomorphic diffeomorphism of $\mathfrak p_+$ onto a dense open subset of $M$, which will be used as a chart on $M$.

These two realizations of a Hermitian symmetric space of compact type are related through the \emph{Harish Chandra construction} of a Cartan subspace of $\mathfrak p_0$ from a Cartan subalgebra of $\mathfrak g$. Start with a maximal Abelian Lie subalgebra $\mathfrak h_0$ contained in $\mathfrak k_0$. Notice that $\mathfrak h_0\supset \mathfrak c\ni H_0$. Let $\mathfrak h$ be its complexification, which is a Cartan subalgebra of $\mathfrak g$. Let $\Delta= \Delta( \mathfrak g, \mathfrak h)$ be the corresponding root system, which is viewed as a subset of $(i\mathfrak h_0)^*$. If $\alpha$ is a root, then $\alpha(H_0) \in \{-i,0,i\}$. Choose a linear order on $(i\mathfrak h_0)^*$ such that
\[\alpha(H_0) = i\quad  \Longrightarrow \quad \alpha \succ 0\ .
\]
There is a corresponding partition of $\Delta$ 
\[\Delta = \Delta_c\cup \Delta_{nc}^+\cup \Delta_{nc}^-
\]
where $\Delta_c$ is the set of compact roots (those for which $\alpha(H_0) = 0$) and $\Delta_{nc}^+$ (resp. $\Delta_{nc}^- $) the set of positive (resp. negative) non-compact roots (those which satisfy $\alpha(H_0) = i$, resp. $\alpha(H_0) =-i$). For $\gamma$ any positive non-compact root, it is possible to choose elements $X_{\pm \gamma}\in \mathfrak g_{\pm \gamma}$ such that
\[X_\gamma-X_{-\gamma} \in \mathfrak u,\qquad i(X_\gamma+X_{-\gamma}) \in \mathfrak u
\]
normalized such that
\[[X_\gamma, X_{-\gamma}] = \frac{2}{\gamma(H_\gamma)} H_\gamma
\]
where $H_\gamma$ is the unique vector in $\mathfrak h$ satisfying $B(H_\gamma, H) = \gamma(H)$ for all $H\in \mathfrak h$.

\begin{proposition} There exists a set $\Gamma = \{ \gamma_1,\dots, \gamma_r\}$  of strongly orthogonal non-compact positive roots such that
\[\mathfrak a_0 = \sum_{j=1}^r \mathbb R (X_{\gamma_j}+X_{-\gamma_j})
\]
is a Cartan subspace of the pair $(\mathfrak g_0, \mathfrak k_0)$.
\end{proposition}

It will be useful to also introduce
\[\mathfrak a_+= \sum_{j=1}^r \mathbb C X_{\gamma_j} \subset \mathfrak p_+\ .
\]
As $X \longmapsto \frac{1}{2}(X-iJX)$ is an $\Ad K_0$-covariant real isomorphism of $\mathfrak p_0$ onto $\mathfrak p_+$, it follows that
\begin{equation}\label{aKp}
\cup_{k\in K_0} \Ad k\, (\mathfrak a_+) = \mathfrak p_+\ .
\end{equation}

\section{Fine Riemannian geometry of a CHSS}

\subsection{ Helgason spheres}

Let $N$ be an irreducible Riemannian (not necessarily Hermitian) symmetric space of compact  type. Let $\kappa>0$ be the maximum of the sectional curvatures of $N$. Then a \emph{Helgason sphere} is a totally geodesic submanifold of $N$, of constant curvature $\kappa$ and of maximal dimension among submanifolds of this type. The Helgason spheres are conjugate under the isometry group of $M$. See \cite{hsphere} for more details.

When $M$ is a compact irreducible Hermitian symmetric space, a Helgason sphere turns out to be  a complex submanifold of $M$, of complex dimension 1, and of constant curvature equal to the maximum of the holomorphic sectional curvature $\kappa$ and bihomorphically isomorphic to the Riemann sphere $\mathbb C\mathbb P_1$.

Let us recall the construction of a Helgason sphere in this special case. Let $\mathfrak a_0$ be a Cartan subspace of $\mathfrak p_0$ and let $r=\dim \mathfrak a_0 =\text {rank } M$. Clearly, the space $i\mathfrak a_0$ is a Cartan subspace for the pair $(\mathfrak u, i\mathfrak p_0)$. The system $\Sigma$ of restricted roots of the pair $(\mathfrak u, i\mathfrak a_0)$ is known to be of type $C_r$ or $BC_r$. Hence there a basis of $i\mathfrak a_0^*$, say $\Gamma=\{\gamma_1,\dots \gamma_r\}$, where $r$ is the \emph{rank} of $M$ such that  
 the roots are given by
\[\pm\gamma_k, 1\leq k \leq r,\qquad \text{with multiplicity 1 } \]
\[\pm\frac{1}{2}\gamma_k\pm\frac{1}{2}\gamma_l,\ 1\leq k<l\leq r,\qquad \text{with multiplicity } a
\]
\[\text{ and possibly } \pm\frac{1}{2} \gamma_k, 1\leq k\leq r, \qquad \text{with even multiplicity } 2b \ . 
\]
Now extend the Cartan subspace $i\mathfrak a_0$ to a Cartan subalgebra of $\mathfrak u$ , say $\mathfrak h_0= \mathfrak b_0+i\mathfrak a_0$ with $\mathfrak b_0\subset \mathfrak k_0$. Let $\mathfrak h$ be its complexification and let  $\Delta = \Delta(\mathfrak g, \mathfrak h)$ be the associated root system. For $1\leq k\leq r$, $\gamma_k$ is of multiplicity 1, hence is the restriction to $i\mathfrak a_0$ of a unique root $\widetilde \gamma_k\in \Delta$ . The Lie algebra $\mathfrak g^{(k)}$ generated by the root spaces $\mathfrak g_{\pm {\widetilde \gamma}_k}$ is isomorphic to ${\mathfrak sl}(2,\mathbb C)$, stable by $\tau$ and $\mathfrak{u}^{(k)}={\mathfrak g^{(k)}}^\tau$ is isomorphic to $
\mathfrak {su}(2)$. Let $\mathbf G^{(k)}$ be the closed analytic subgroup of $\mathbf G$ with Lie algebra $\mathfrak g_k$. Then  $U^{(k)}={\big(\mathbf G^{(k)}\big)}^\tau$ is a  maximal compact subgroup of $\mathbf G^{(k)}$.  The orbit of $o$ under $\mathbf G^{(k)}$ coincides with its orbit under $U^{(k)}$ and is isomorphic to a Riemann sphere $\mathbb C \mathbb P_1$. It is in fact is a Helgason sphere (see \cite{hsphere,tak}).
 
There is another way to obtain a Helgason sphere, which we will use in Section 3. This time, we  choose a Cartan subalgebra $\mathfrak h_0 \subset \mathfrak k_0$, and we already noticed that its complexification $\mathfrak h$ is a Cartan subalgebra of $\mathfrak g$. Let $\Delta$ the corresponding  root system, and let $\Pi$ be the set simple roots with respect to some choice of positive roots. Among the simple roots, one and only one is non-compact, say $\alpha_1$. Let $\mathfrak g_{\pm\alpha_1} $ be the corresponding root spaces. Then, as $\tau(iH) = -iH$, $\tau(\mathfrak g_{\alpha_1}) = \mathfrak g_{-\alpha_1}$. The Lie algebra generated by $\mathfrak g_{\alpha_1}$ and $\mathfrak g_{-\alpha_1}$ is isomorphic to $\mathfrak {sl}(2,\mathbb C)$ and stable by $\tau$. As in the previous construction, the corresponding analytic subgroup $\mathbf G^{(\alpha_1)}$ of $\mathbf G$ is closed and stable by $\tau$. The orbit of $o$ under $\mathbf G^{(\alpha_1)}$ is again a Helgason sphere. It can be shown by using the Harish Chandra construction of a Cartan subalgebra in $\mathfrak p_0$ from $\mathfrak h$, followed by a \emph{Cayley transform} (see \cite{kor} III.2). After these operations, we are back to the first construction and the conclusion follows. This last construction has a more general version, valid for any symmetric $R$-space, presented  in \cite{tak}.

\subsection{The polysphere embedding}

Going back to notations of the beginning of subsection 2.1,  observe that as the roots $\gamma_k, 1\leq k\leq r$ are mutually strongly orthogonal, the subalgebras $\mathfrak g^{(k)}$ mutually commute to each other. Form
\begin{equation}\label{Gamma}
\mathfrak g(\Gamma) = \bigoplus_{k=1}^r \mathfrak g^{(k)},\qquad \mathfrak u(\Gamma) = \bigoplus_{k=1}^r \mathfrak u^{(k)}\ ,
\end{equation}
and let $\mathbf G(\Gamma)$ and $U(\Gamma)$ be the corresponding analytic subgroups of $\mathbf G$. 

\begin{proposition} [polysphere embedding]\label{poly}The orbit of $o$ under the action of $U(\Gamma)$ is a complex totally geodesic submanifold $\mathbf S(\Gamma)$ of $M$, isomorphic to $S^{(1)}\times \dots \times S^{(k)}\times \dots \times S^{(r)}$, where for $1\leq k\leq r$, $S^{(k)}$ is a Helgason sphere, isomorphic to $\mathbb C \mathbb P_1$. 
\end{proposition}
See \cite{wo}. Let 
\[\mathfrak a_0^{(k)} = \mathfrak a_0\cap \mathfrak g^{(k)},\qquad\mathfrak a_+^{(k)} = \mathfrak a_+\cap \mathfrak g^{(k)}\ ,
\]
and observe that $\displaystyle \mathfrak a_0 = \bigoplus_{k=1}^r \mathfrak a_0^{(k)}$, and
\[ \qquad T(\Gamma) = \exp(i\mathfrak a_0) = \prod_{k=1}^r \exp\big(i\mathfrak a_0^{(k)}\big)\]is a maximal torus $T(\Gamma)$ in $M$.
For further reference, let   $o=(o_1,\dots, o_r)$ where, for $1\leq k\leq r$,  $o_k$ is the origin point in $S^{(k)}$.

\begin{corollary}\label{geodesicpolysphere}
 Let $\gamma : [0,1] \longrightarrow M$ be a geodesic curve starting from $\gamma(0) = p$. Then there exists an element $u\in U$ such that $u(p)=o$ and $u\circ \gamma$
is  contained in $\mathbf S(\Gamma)$.
\end{corollary}

\begin{proof}
First, as $U$ is transitive on $M$, there is an element of $U$ which maps $p$ to the origin $o$. Now any geodesic curve through $o$ can be mapped by an element of $K_0$ to a geodesic curve contained in $T(\Gamma)\subset \mathbf S(\Gamma)$. The lemma follows by composing the two elements of $U$.
\end{proof}

\subsection{First conjugate locus}

Let $N$ be a compact Riemannian manifold. For  $p\in N$ let $T_pN$ be the tangent space to $N$ at $p$, and let $\Exp_p : T_pN\longrightarrow N$ be the exponential map  with source $p$. The \emph{tangent 
conjugate locus} of $p$ is the space $C_p\subset T_pN$ defined by
\[X\in C_p \Longleftrightarrow d\Exp_p(X) \text{ is singular} \ .
\]
The \emph{tangent first conjugate locus} of $p$ is  the subset $C_p^{(1)}$ defined as
\[C_p^{(1)} = \{ x\in C_p, tX\notin C_p, \text{ for any } t, 0\leq t<1\}\ .
\]
The \emph{conjugate locus} $\mathbf C_p$  (resp. \emph{first conjugate locus} $\mathbf C_p^{(1)})$ of $p$ is the image under the exponential map with source at $p$ of the tangent conjugate locus (resp. first tangent conjugate locus).

For an irreducible Riemannian symmetric space of compact  type, it is enough to determine the tangent conjugate locus at the origin $o$. Its description is known (see \cite{h} Ch. VII, Prop. 3.1). In our situation, with the notation introduced above, first

\[C_o = \cup_{k\in K} \Ad k\,(C_o\cap  i\mathfrak a_0)\ .
\] 
Further,  
\[X\in C_o\cap \mathfrak i\mathfrak a_0 \quad \Longleftrightarrow\quad \exists \alpha\in \Sigma, \alpha(X) \in i\pi\left(\mathbb Z\smallsetminus\{0\}\right)\ .
\]

\begin{proposition}\label{FTCL}
 Let $M$ be a Hermitian symmetric space of compact type. The first tangent conjugate locus at the origin $o$ is given by
\[C_o^{(1)} = \cup_{k\in K} \Ad k\, (C_o^{(1)}\cap i\mathfrak a_0)\ ,\qquad C_o^{(1)}\cap i\mathfrak a_0 = \{ X\in i\mathfrak a_0, \max_{1\leq j\leq r} \vert\gamma_j(X)\vert = \pi\}\ .
\]

\end{proposition}
\begin{proof}
For any restricted root $\alpha$
\[\vert\alpha(X)\vert\leq \max_{1\leq j\leq r} \vert \gamma_j(X)\vert,
\]
so that 
\[\max_{\alpha\in \Sigma} \vert \alpha(X)\vert= \max_{1\leq j\leq r} \vert \gamma_j(X)\vert\ ,
\]
and the proposition follows.
\end{proof}
We may now combine this result with the polysphere embedding (cf. Proposition \ref{poly}).

\begin{proposition}\label{FCL}
 Let $q \in M$. Then $q$ belongs to the first conjugate locus of $o$ if and only if there exists a polysphere $\mathbf S(\Gamma)\simeq S^{(1)}\times \dots \times S^{(r)}$ such that $q=(q_1,\dots q_r)$ and there exists some $j, 1\leq j\leq r$ such that $q_j$ is antipodal to $o_j$.
\end{proposition}

\begin{proof} By using the action of $K_0$, we may assume that  $q\in T(\Gamma)\subset \mathbf S(\Gamma)$.  Then $q$ is in the first conjugate locus of $o$ if and only if there exists $H\in \mathfrak a_0$ such that $\exp iH(o) = q$ and there exists $j, 1\leq j\leq r$ such that $\gamma_j(H) = \pm \pi$. Let $q=(q_1,\dots, q_r)$ with $q_k\in S^{(k)}$, and let $H=\sum_{k=1}^r H_k$ be the decomposition of $X$ with respect to  the decomposition $\mathfrak a_0 = \bigoplus_{k=1}^r  \mathfrak a_0^{(k)} $. Now $\gamma_j(H_j) = \gamma_j(H)=\pm \pi$, hence $q_j = \exp iH_j (o_j) = q_j$ is in the first conjugate locus of $o_j$ in the sphere $S^{j}\simeq \mathbb C\mathbb P_1$, that is to say $q_j$ is antipodal to $o_j$.
\end{proof}

\subsection{Cut locus}
Let $N$ be a Riemannian manifold and let $p\in N$. A tangent vector $X$ at $p$ is said to belong to the \emph{tangential cut locus}  of $p$ is the geodesic segment starting from $o$ in the direction $X$ is arc-length minimizing up to $\Exp_p X$, but not further.

 The \emph{cut locus} of $p$ is the image under $\Exp_p$ of the tangent cut locus of $p$.
 
\begin{proposition}\label{Cut}
 Let $M$ be an irreducible  Hermitian symmetric space of compact type. Then for any $p$ in $M$, the cut locus of $p$ is equal to the first conjugate locus of $p$. 
\end{proposition}

\begin{proof} Any Hermitian symmetric space of compact type is simply connected (see \cite{h} chapter VIII, Theorem 4.6). Hence the statement follows from a more general theorem due to Crittenden (see \cite{cr} Theorem 5).
\end{proof}

As a consequence, Proposition \ref{FCL} is valid when the first conjugate locus is replaced by the cut locus. Let us mention that another proof of this result can be obtained for Hermitian symmetric spaces of compact type by using the more sophisticated result obtained by Tasaki, based on the notion of \emph{diastasis} (see \cite{t} Corollary 8).

\subsection{Unique minimizing geodesic}

Given two points $p,q\in M$, there always exists a minimizing geodesic segment joining $p$ to $q$. The next proposition describes the situation where the segment is not unique.

\begin{proposition}\label{UMG}
 Let $p,q\in M$. The following assertions are equivalent :
\smallskip

$i)$ there exists (at least) two distinct minimizing geodesic segments joining $p$ to $q$.
\smallskip

$ii)$ there exists an infinity of distinct minimizing geodesic segments joining $p$ to $q$.
\smallskip

$iii)$ $q$ belongs to the cut locus of $p$.
\smallskip

$iv)$ there exists $u\in U$ such that $u(p) = o$, $q'=u(q)$ belongs to $\mathbf S(\Gamma)$ and for some  $j, 1\leq j\leq r$, $q'_j$ is the antipodal point of $o_j$  in $S^{(j)}$.
\end{proposition}
\begin{proof} The fact that $i)$ implies $iii)$ is a general result valid for any  Riemannian manifold (see e.g. \cite{dc} Proposition 2.2, chap. 13). Now let $q$ belong to the cut locus of $p$. Combine Proposition \ref{Cut} and Proposition \ref{FCL} to conclude that $iii)\Longrightarrow iv)$. As two antipodal points on a sphere can be joined by an infinite number of minimizing distinct geodesic segment, $iv)$ implies $ii)$ and $ii)$ trivially implies $i)$.
\end{proof}



\subsection{The big cell}

We now use the realization of $M$ as $\boldsymbol G/\boldsymbol K\boldsymbol P_-$ and consider the mapping
\[\Xi : \mathfrak p^+ \longrightarrow \boldsymbol G/\boldsymbol K\boldsymbol P_-,\quad z\longmapsto \exp(z) \boldsymbol K\boldsymbol P_-\ .
\]
Then the image of $\Xi$ is a dense open subset of $M$. 

\begin{proposition}\label{bigcell}
The image of $\,\Xi$ is equal to $M\smallsetminus \mathbf C_0^{(1)}$.
\end{proposition}
\begin{proof} First remark that both $\im (\Xi)$ and  $M\smallsetminus \mathbf C_0^{(1)}$ are invariant under the action of $K$. 

Next recall the Lie algebra $\mathfrak g(\Gamma)$ introduced in \eqref{Gamma}, which is isomorphic to a product of $r$ copies of $\mathfrak {sl}(2,\mathbb C)$ , and 
let 
\[{\mathfrak p_+} ^{(k)}= \mathfrak p_+\cap \mathfrak g^{(k)},\qquad  \mathfrak p_+(\Gamma) = \mathfrak g(\Gamma) \cap \mathfrak p_+=\bigoplus_{j=1}^r {\mathfrak p_+}^{(k)}\ ,\]
and consider the restriction $\Xi_{\vert {\mathfrak p_+}^{(k)}}$ of $\Xi$ to ${\mathfrak p_+}^{(k)}$. By an elementary $SL(2,\mathbb C)$ calculation, the image of $\Xi_{\vert {\mathfrak p_+}^{(k)}}$ is equal to $S^{(k)}\smallsetminus \{ -o_k\}$, where $-o_k$ is the antipodal point to $o_k$ in $S^{(k)}$. Hence the image of $\Xi_{\vert \mathfrak p_+(\Gamma)}$ is equal to \[\big(S^{(1)}\smallsetminus \{-o_1\} \big)\times \dots \times \big(S^{(r)}\smallsetminus \{-o_r\}\big)\ ,\]
and is equal to
 $ \mathbf S(\Gamma)\smallsetminus \mathbf C_0^{(1)}$. By $K$-invariance, $\im(\Xi) = M\smallsetminus \mathbf C_0^{(1)}$.
 \end{proof}

\section{The homology of  Hermitian symmetric spaces of compact type}

To understand the ambiguity in the definition of the symplectic area of a geodesic triangle, it is necessary to study its topology and more precisely the homology with coefficients in $\mathbb Z$ in degree 2. In general, the homology 
$H_*(M,\mathbb Z)$ is explicitly known and related to \emph{Schubert cells}. Our main reference for this section is \cite{bgg}.

Let $\mathfrak h$ be a complex Cartan subalgebra of $\mathfrak k$. Notice that $\mathfrak z\subset \mathfrak h$ and hence $\mathfrak h$ is also a Cartan subalgebra of $\mathfrak g$. Let $\Delta$ be the root system of $(\mathfrak g, \mathfrak h)$ and choose a system of positive roots $\Delta^+$ such that any $\alpha\in \Delta$ satisfying $\alpha(H_0)=i$ is a positive root. Let $\mathbf B^-$ be the Borel subgroup of $\mathbf G$ associated to $-\Delta^+$, and let $\mathbf N^-$ be its unipotent radical. Let $\mathbf W$ be the Weyl group of $\mathbf G$. For $w\in \mathbf W$, let $l(w)$ be the length with respect to the set $\Pi$ of simple roots in $\Delta^+$. Let $s$ be the unique element of maximal length $l(s) = r$, and let $\mathbf N = s\mathbf N^-s^{-1}$. For $w\in \mathbf W$, set
\[{\mathbf N}^-_w = w\mathbf N w^{-1} \cap {\mathbf N^-}\ .
\]
Then $\mathbf N^-_w$ is a unipotent subgroup of $\mathbf N^-$ of (complex) dimension $l(w)$.

Among the simple roots, one and only one is non-compact, say $\alpha_1$. In particular, $\alpha_1(H_0) = i$. Let $\Theta = \Pi\smallsetminus{\{\alpha_1\}}$ be the set of compact simple roots, and let $\mathbf W_\Theta$ be the subgroup of $\mathbf W$ generated by the reflexions $s_\alpha, \alpha \in \Theta$. $\mathbf W_\Theta$ is the Weyl group of the reductive subgroup $\mathbf K$.

The Borel subgroup $\mathbf B^-$ acts on $M$ with a finite number of orbits. They are indexed by the coset space $\mathbf W/\mathbf W_\Theta$. More precisely,
\[M = \cup_{w\in \mathbf W/\mathbf W_\Theta} \mathbf B^- w o\ .
\]

\begin{lemma} Let $\mathbf W^\Theta$ be the set of $w\in \mathbf W$ such that $w\Theta\subset \Delta_+$. Then any coset in $\mathbf W/\mathbf W_\Theta$ contains a unique representative in $\mathbf W^\Theta$.
\end{lemma}

For $w\in \mathbf W^\Theta$, let $X_w = \mathbf B^- w o$. 

\begin{lemma} Let $w\in W^\Theta$.  The map
\begin{equation}\label{parSchubert}
\mathbf N^-_w \ni n\longmapsto nwo \in X_w
\end{equation}
is an isomorphism of algebraic varieties.
\end{lemma}

Denote by $[\overline {X_w}]$ the image in $H_*(M,\mathbb Z)$ of the fundamental cycle of $\overline {X_w}$ under the mapping induced by $\overline {X_w} \longmapsto M$.

\begin{theorem} The elements $[\overline {X_w}], w\in W^\Theta$ form a free basis of $H_*(M,\mathbb Z)$.
\end{theorem}
See \cite{bgg}, Proposition 5.2, attributed to A. Borel (see \cite{bo}).

\begin{lemma} The variety $X_w$ is of complex dimension 1 if and only if $w=s_{\alpha_1}$.
\end{lemma}
\begin{proof}
As $\dim(X_w)$ is equal to $l(w)$, it suffices to show that the reflexion $s_\alpha$ for $\alpha$ a simple root belongs to $\mathbf W^\Theta$ if and only if $\alpha = \alpha_1$. Now if $\alpha$ is a simple compact root, $s_\alpha \alpha = -\alpha$ so that $s_\alpha(\Delta_\Theta) \nsubseteq \Delta^+$. On the other hand $s_{\alpha_1}$ transforms any positive compact root into a positive root, so that $s_{\alpha_1}(\Delta_\Theta)\subset \Delta^+$.
\end{proof}

\begin{proposition}\label{sphere}
 The variety $\overline {X_{s_{\alpha_1}}}$ is  a Helgason sphere, isomorphic to $\mathbb C \mathbb P_1$.
\end{proposition} 
\begin{proof} As $\alpha_1$ is a simple root, $s_{\alpha_1}$ permutes $\Delta_+\smallsetminus \{ \alpha_1\}$ and maps $\alpha_1$ to $-\alpha_1$, so that 
 \[ N^-_{s_{\alpha_1}} = \exp\big(\mathfrak g_{-\alpha_1}\big)\ .
\]
Let $\mathfrak g^{(\alpha_1)}$ be the Lie algebra (isomorphic to $\mathfrak {sl}(2,\mathbb C)$) generated by $\mathfrak g_{\alpha_1}$ and $\mathfrak g_{-\alpha_1}$, and let $\mathbf G^{(\alpha_1)}$   be the corresponding analytic subgroup of $\mathbf G$. Then as explained in subsection 2.1, the orbit $\mathbf G^{(\alpha_1)}o$ is a Helgason sphere. Now
\[N^-_{s_{\alpha_1}} s_{\alpha_1} o = s_{\alpha_1} \exp\big(\mathfrak g_{\alpha_1}\big) o\ ,
\]
and $\exp\big(\mathfrak g_{\alpha_1}\big) o$ is dense in $\mathbf G^{(\alpha_1)}o$. Hence 
 \[ \overline {X_{s_{\alpha_1}}} =\overline{N^-_{s_{\alpha_1}} s_{\alpha_1} o}= s_{\alpha_1}\mathbf G^{(\alpha_1)} o = \mathbf G^{(\alpha_1)} o\] 
 and the proposition follows.
\end{proof}
Theorem 3.1 can now be formulated in degree 2.
\begin{theorem} The group $H_2(M,\mathbb Z)$ is equal to $\mathbb Z [\overline {X}_{s_{\alpha_1}}]$.
\end{theorem}

\section{The symplectic area of geodesic triangles}

\subsection{The normalized K\"ahler form on $M$}

On $M$ there is a canonical  K\"ahler form $\omega$. As $M$ is assume to be irreducible, $\omega$ is unique up to a constant. The form $\omega$ is related to the invariant Hermitian metric $q$ on $M$ by the relation
\[\omega_m(X,Y)=q_m(JX,Y)
\]
for $X,Y\in T_mM$. As will appear in the sequel, it is preferable to use a different normalization, chosen such that the sectional holomorphic curvature has maximal value $1$. It is sufficient to choose the normalization of $q_o$, and under the identification of $T_oM\simeq i\mathfrak p_0$, the normalization is given by
\begin{equation}\widetilde q_0 = \frac{2}{p}q_0\ ,
\end{equation}
where $p= (r-1)a+b+2$. As similar computations for the dual space $M^d$  were carefully done in \cite{co2}, we skip the details. The maximum of the scalar curvature is reached for the tangent space to the complex Riemann spheres $S^{(j)}$ associated to any root $\gamma_j$ and, as a particular case to the sphere called $\overline {X}_{s_{\alpha_1}}$ in section 3. Notice that the restriction of the renormalized Hermitian metric (and consequently of the renormalized K\"ahler form $\widetilde \omega = \frac{2}{p} \omega$) to such a sphere coincides with the usual metric (resp. area form) on the sphere $\mathbb C\mathbb P_1\simeq S^2$. The following proposition is an immediate consequence of this remark.
\begin{proposition}\label{intS}
 Let $S\simeq \mathbb C\mathbb P_1$ be a Helgason sphere. Then
\begin{equation}
\int_S \widetilde \omega= 4\pi\ .
\end{equation}
\end{proposition}

A triplet $(p_1,p_2,p_3)$ of points in $M$ is said to be \emph{regular} if there exists a unique  minimizing geodesic segment between any two of the three points. Given a regular triplet $(p_1,p_2,p_3)$ we can form without ambiguity the oriented triangle $T(p_1,p_2,p_3)$, by joining $p_1$ to $p_2$ via the unique minimizing geodesic from $p_1$ to $p_2$ and similarly for $(z_1,z_2)$ and $(p_3,p_1)$. Let $\Sigma=\Sigma(p_1,p_2,p_3)$ be any $2$-simplex whose boundary is the oriented triangle $T(p_1,p_2,p_3)$. 
Then define 
\[A\big(\Sigma\big) = \int_{\Sigma(p_1,p_2,p_3)} \widetilde \omega \ .
\]
Now, by Stokes theorem, as $\omega$ is a closed form, this integral does not change if the simplex $\Sigma$ is changed in a smooth way. However, as $\omega$ is not exact, the corresponding global statement is not true. 

\begin{theorem} Let two 2-simplices $\Sigma_1$ and $\Sigma_2$, both having $T(p_1,p_2,p_3)$ as boundary. Then
\[A(\Sigma_1) \equiv A(\Sigma_2) \mod 4\pi\ .
\]
\end{theorem}

\begin{proof}
Consider two simplices $\Sigma_1$ and $\Sigma_2$ whose boundary is the oriented triangle $T(p_1,p_2,p_3)$. Then the boundary of $\Sigma_1$ and $\Sigma_2$ is the same, which we can translate in homological terms 	as
\[\partial\big(\Sigma_1 - \Sigma_2 \big)=0,\]
where $\partial$ is the boundary operator. Hence $\Sigma_1 - \Sigma_2$ is a 2-cycle with integer coefficients, which by our computation of $H_2(M,\mathbb Z)$ implies that the corresponding homology class is given by\[[\Sigma_1 -\Sigma_2] = n [S]
\]
for some $n\in \mathbb Z$. Hence, as $\omega$ is closed,
\[\int_{\Sigma_1} \widetilde \omega -\int_{\Sigma_2}\widetilde  \omega = n \int_S \widetilde \omega = n (4\pi)\ ,
\]
and the statement of the theorem follows.
\end{proof}

For $(p_1,p_2,p_3)$ a regular triplet, form the geodesic triangle $\mathcal T(p_1,p_2,p_3)$. Let $\Sigma$ be a 2-simplex whose boundary is $\mathcal T(p_1,p_2,p_3)$. Then the quantity
$ \displaystyle e^{\frac{i}{2} \int_\Sigma \widetilde \omega}$ 
is independent of the simplex $\Sigma$ and thus defines a $U(1)$-valued function $\Psi(p_1,p_2,p_3)$ depending only of $(p_1,p_2,p_3)$.

\begin{theorem} The function $\Psi$ defined on regular triplets by
\[
\Psi(p_1,p_2,p_3) = e^{\frac{i}{2} \int_{\Sigma} \widetilde \omega}
\]
where $\Sigma$ is a 2-simplex whose boundary is the oriented geodesic triangle $T(p_1,p_2,p_3)$ satisfies the following properties

$i)$ for any permutation $\sigma$ of $\{1,2,3\}$
\[\Psi\big(z_{\sigma(0)}, z_{\sigma(1)}, z_{\sigma(2)}\big) = \Psi\big(z_0,z_1,z_2\big)^{\epsilon(\sigma)},\] where $\epsilon(\sigma)$ is the signature of the permutation $\sigma$.

$ii)$ for any quadruplet $(p_0,p_1,p_2,p_3)$ such that the four triplets $(p_0,p_1,p_2)$, $(p_1,p_2,p_3)$, $ (p_2,p_3,p_0)$ and $(p_3,p_0,p_1)$ are regular\[\Psi(p_0,p_1,p_2)  \Psi(p_1,p_2, p_3)^{-1} \Psi(p_2,p_3,p_0) \Psi(p_3,p_0,p_1)^{-1} = 1\ .
\]
\end{theorem}
\begin{proof}
For $i)$, observe that the triangle $\mathcal T\big(p_{\sigma(0)}, p_{\sigma(1)},p_{\sigma(2)}\big)$ has the same orientation as $T(p_0,p_1,p_2)$ when $\sigma$ is an even permutation and the opposite orientation if $\sigma$ is an odd permutation.
Hence $\Sigma$ is a 2-simplex with boundary equal to $\mathcal T\big(p_{\sigma(0)}, p_{\sigma(1)},p_{\sigma(2)}\big)$ if $\sigma$ is even and to its opposite if $\sigma$ is odd. The values of $\Psi\big(p_{\sigma(0)}, p_{\sigma(1)}, p_{\sigma(2)}\big)$ is equal to $\Psi\big(p_0,p_1,p_2\big)$ if $\sigma$ is even and to $\Psi\big(p_0,p_1,p_2\big)^{-1}$ if $\sigma$ is odd.

For $ii)$ choose four 2-simplices $\Sigma(p_0,p_1,p_2), \dots, \Sigma(p_3,p_0,p_1)$ such that the boundary of $\Sigma(p_0,p_1,p_2)$ is equal to the oriented triangle $\mathcal T(p_0,p_1,p_2)$, \dots , the boundary of $\Sigma(p_3,p_0,p_1)$ is equal to the oriented triangle $\mathcal T(p_3,p_0,p_1)$. Then observe that 
\[\partial\big(\Sigma(p_0,p_1,p_2)-\Sigma(p_1,p_2,p_3) +\Sigma(p_2,p_3,p_0) -\Sigma(p_3,p_0,p_1)\big) =0\ .
\]
Hence
\[\int_{\Sigma(p_0,p_1,p_2)}\!\!\!\!\!\!\!\!\widetilde \omega- \int_{\Sigma(p_1,p_2,p_3)} \!\!\!\!\!\!\!\!\widetilde \omega+\int_{\Sigma(p_2,p_3,p_0)} \!\!\!\!\!\!\!\!\widetilde \omega -\int_{\Sigma(p_3,p_0,p_1)}\!\!\!\!\!\!\!\!\widetilde \omega\equiv 0 \mod 4\pi\ .
\]
The identity $ii)$ follows from this last equation.

\end{proof}

\section{The canonical kernel and the K\"ahler potential}

The symplectic form $\omega$ on $M$ is closed, but not exact. Hence, it is not possible to find  a $1$-form $\rho$ such that $\omega = d\rho$ and hence it is also impossible to find a global K\"ahler potential, i.e. a function $k(z)$ on $M$ such that $\omega = i\partial \overline{\partial} k$. We are forced to restrict to an open set of $M$, which is topologically trivial. 
It seems reasonable to use the chart $(\Xi, \mathfrak p_+)$. The procedure 
leading to the construction of the K\"ahler potential follows closely the construction in the non-compact case, as presented in \cite{s}, chapter II.

\subsection{The canonical kernel for the compact space}

The main ingredient is the \emph{automorphy kernel}. We first recall its definition  in the case of a noncompact Hermitian symmetric domain. We essentially follow \cite{s}.

Let $\sigma$ be the involution of $\mathfrak g$ with respect to $\mathfrak g_0=\mathfrak k_0\oplus \mathfrak p_0$ and denote by $\sigma$ its lift to $\mathbf G$. 

The map
 \[ \mathbf P_+\times\mathbf K\times \mathbf P_-\ni(p_+,k,p_-)\quad\longmapsto \quad p_+kp_-
\]
is a diffeomorphism on a dense open subset of $\mathbf G$, and we write the inverse map as
$g=g_+\,g_0\,g_-$ when 
$g\in \mathbf P_+\mathbf K\mathbf P_-$. We can define a partial action of $\mathbf G$ on $\mathfrak p_+$ by
\[g(z) = \big(g \exp(z)\big)_+,
\]
which is nothing else than the expression in the chart $(\Xi, \mathfrak p_+)$ of  the action of $\mathbf G$ on $M$. Where defined, the differential $J(g,z)$ of the map $z\longmapsto  g(z)$ is given by
\[J(g,z) = \big(g\exp(z)\big)_0\ .
\]

For $z,w\in \mathfrak p_+$ such that
\[\exp(-\sigma w)\exp(z) \in \mathbf P_+\mathbf K\mathbf P_-
\]
 the \emph{canonical automorphy kernel} $K(z,w)$ is defined by 
\[K(z,w) = J(\exp(-\sigma w), z)^{-1} = \big(\exp(-\sigma w)\exp(z)\big)_0^{-1}\ .
\]
To pass to the case of a compact Hermitian symmetric space, we replace $\sigma$ by $\tau$. 

For $z,w$ in $\mathfrak p_+$ such that $\exp(-\tau w) \exp(z)$ belongs to 
$\mathbf P_+\mathbf K\mathbf P_-$, define the \emph{ compact automorphy kernel} $K_c(z,w)$ as
\[K_c(z,w) = \big(\exp(-\tau w) \exp(z)\big)_0^{-1}\ .
\]
An elementary but crucial observation is that on $\mathfrak p^+$, $\tau$ and $\sigma$ differ by a sign.  So for $z,w\in \mathfrak p_+$  where $K_c$ is defined,
\[K_c(z,w) = K(z,-w)\ .
\]
The results to follow are stated without proofs, as they are mostly consequence of this remark and can be deduced from the analogous result in the non-compact case.
\begin{proposition}
 Let $(z,w)\in \mathfrak p_+$ such that $K_c$ is defined at $(z,w)$ and let $g\in \mathbf G$ such that $g$ is defined at $z$ and $\tau(g)$ is defined at $w$. Then $K_c$ is defined at $\big(g(z),\tau(g)(w)\big)$ and
 \[K_c\big(g(z), \tau(g)(w)\big) = J(g,z)K_c(z,w) \tau\big( J(\tau(g),w)\big)^{-1}\ .
 \]
\end{proposition}

The automorphy kernel allows to express the Riemannian metric $q$ on $M$ in the chart $(\Xi,\mathfrak p_+)$.
\begin{proposition} For $z\in \mathfrak p_+, \zeta, \eta\in \mathfrak p_+$
\[q_z(\zeta, \eta) = -\frac{1}{2} B\big(\Ad K_c(z,z)^{-1} \zeta,\tau\eta\big)\ .
\]
\end{proposition}
\begin{proposition} Let $(z,w)\in \mathfrak p_+$ such that $K(z,w)$ is defined. Then on $\mathfrak p_+$
\[\Ad_{\mathfrak p^+} K_c(z,w) = \id -\ad[z,\tau w]+\frac{1}{4} (\ad z)^2(\ad\tau w)^2\ .
\]
\end{proposition}

Next let for $z,w\in \mathfrak p_+$
\[k_c(z,w) = \det \Ad_{\mathfrak p_+} K_c(z,w) = k(z,-w)
\]

Define for $(z,w)\in \mathfrak p_+$
\[k_c(z,w) = \det \Ad_{\mathbf p_+} K_c(z,w),
\]
where $\det$ refers to the complex determinant, and observe that the definition makes sense for \emph{all} $(z,w)\in \mathfrak p_+$. The covariance property of the automorphy kernel implies the following covariance property for the kernel $k_c$.
\begin{proposition}Let $z,w\in \mathfrak p_+$ and let $g\in \mathbf G$ such that $g$ id defined at $z$ and $\tau(g)$ is defined at $w$. Then
\begin{equation}\label{covkc}
k_c\big(g(z),\tau(g)(w)\big) = j(g,z)\, k_c(z,w)\,\overline { j\big(\tau(g),w\big)}\ .
\end{equation}

\end{proposition}

The kernel $k_c$ can be explicitly expressed on $\mathfrak a_+\times \mathfrak a_+$. 
\begin{proposition} Let $z=\sum_{j=1}^r z_jX_j, w= \sum_{j=1}^rw_jX_j$.
Then
\begin{equation}\label{ksuraplus}
k_c(z,w) = \prod_{j=1}(1+z_j\overline w_j)^p\ .
\end{equation}
\end{proposition}
The last proposition shows in particular that $k_c(z,z)>0$ for $z\in \mathfrak p_+$.
In turn, the kernel $k_c$ allows to describe a K\"ahler potential for $\omega$ on $\mathfrak p_+$.
\begin{proposition} For $z\in \mathfrak p_+$,
\begin{equation}\label{potential}
\omega_z = i\partial \overline{\partial} \log k_c(z,z)\ .
\end{equation}
\end{proposition}

The ingredients for the proof of the analog of the Domic-Toledo formula are now available. But a new difficulty occurs, namely the formula ought to involve an argument of $k_c(z,w)$ for $z,w\in \mathfrak p_+$. But $k_c(z,w)$ may vanish (see \eqref{ksuraplus}). Let us consider first the situation on the simplest example, namely the case where $M=\mathbb C\mathbb P_1$.

\subsection{The case of the projective space}

Let $M=\mathbb C\mathbb P_1$ be the complex projective space, i.e. the space of complex lines in $\mathbb C^2$. As origin, choose the origin  the complex line generated by $(1,0)$. The map 
\[\mathbb C \ni z\quad \longmapsto \quad \text{the complex line generated by } (z,1)\in M
\]
gives a local chart on $M$ and actually coincides with $\Xi$. As usual, it is convenient to consider $M$ as $\mathbb C\cup \infty$ where $\infty$ is  the line  generated by $(0,1)$. 

The group  $\mathbf= SL(2,\mathbb C)$ acts naturally (projectively) on $M$, and the expression of this action in  the chart is given by
\[g=\begin{pmatrix} a&b\\c&d\end{pmatrix}\in SL(2,\mathbb C), \quad z\in \mathbb C, \qquad g(z) = \frac{az+b}{cz+d},
\]
extended to $\mathbb C\cup \infty$ by
by
\[g(\infty) = \frac{a}{c},\quad g\left(-\frac{d}{c}\right)=\infty\quad \text{if } c\neq 0, \quad g(\infty)=\infty \quad \text{if } c=0\ .
\]
Its standard maximal compact subgroup is
 \[U=SU(2)= \left\{ \begin{pmatrix} \alpha&\beta\\-\overline \beta&\overline \alpha\end{pmatrix}, \alpha,\beta\in \mathbb C, \vert \alpha\vert^2+\vert \beta\vert^2=1\right\}\ .\]
 The stabilizer of the origin in $U$ is the subgroup 
 \[K_0 = \left\{\begin{pmatrix} e^{i\frac {\theta}{2}}&0\\0&e^{-i\frac{\theta}{2}}\end{pmatrix}, \theta\in \mathbb R \right\}\ .
 \]
Let $\beta\in \mathbb C, \vert \beta\vert = 1$. Consider the one-parameter subgroup of $SU(2)$ given by
\[g_\beta(t) = \begin{pmatrix} \cos t&\beta \sin t\\ -\overline \beta \sin t&\cos t\end{pmatrix} = \exp t\begin{pmatrix}0&\beta\\-\overline \beta&0 \end{pmatrix}\ .
\]
Then $t\longmapsto g_\beta(t)0 = (\tan t) \beta$ is the expression in the chart of the geodesic curve through $0$ with tangent vector at $0$ equal to $\beta$. 

\begin{lemma} {\ }

$i)$ the antipodal point of $0$ is $\infty$

$ii)$ the antipodal point of $z\neq 0$ is equal to $\displaystyle -{\overline z}^{-1}$.
\end{lemma}

\begin{proof}  Using the notation introduced above, for $t=\pi$ we get $g_\beta(\pi) = \infty$ for any $\beta$, so that $i)$ follows. Next let \[g= \begin{pmatrix} \alpha&\beta\\-\overline \beta&\overline \alpha \end{pmatrix}, \vert \alpha\vert^2 +\vert \beta\vert^2=1\]
be an element of $SU(2)$, so that
\[g(0) = \frac{\beta}{\overline \alpha},\quad g(\infty) = -\frac{\alpha}{\overline \beta}
\]

As $g$ acts by isometrically on  $M$, the antipodal point of $z=\frac{\beta}{\overline \alpha}$ is equal to $-\frac{\alpha}{\overline \beta}= -\overline z^{-1}$.
\end{proof}

The automorphy kernel is given by
\[K(z,w) = \begin{pmatrix}1-z\overline w&0\\ 0&(1-z\overline w)^{-1} \end{pmatrix}
\]
and the  canonical kernel $k(z,w)$ is then given by
\[k(z,w) = (1-z\overline w)^2\ .
\]

The compact canonical kernel is given by
\[k_c(z,w) = (1+z\overline w)^2\ .
\]
In order to be able to define a continuous argument for $k_c$ we introduce the set
\[\mathcal S = \{(z,w)\in \mathbb C\times \mathbb C, 1+z\overline w \notin (-\infty, 0]\}\ .
\]

The geometric significance of the space $\mathcal S$ is given by the following propositions.

\begin{proposition}The pair $(z,w)$ belongs to $\mathcal S$ if and only if there exists a unique minimizing geodesic segment contained in $\mathbb C$ joining $z$ and $w$.
\end{proposition}
\begin{proof}
First observe that $1+z\overline w\in \mathbb R$ is equivalent to $(z,w)$ linearly independent over $\mathbb R$ and hence connected by a line through the origin. Assume now that $1+z\overline w\notin \mathbb R$. Then the antipodal point of $z$ is not equal to $w$, so that there exists a unique minimizing geodesic segment on $M$ from $\Xi(z)$ to $\Xi(w)$. If the segment is not contained in $\mathbb C$, then $\infty$ belongs to this segment. But any geodesic line which passes through $\infty$ passes through its antipodal point $0$, the geodesic would contain $0,z,w$ and so $z,w$ would be linearly dependent over $\mathbb R$.  Hence if $1+z\overline w\notin \mathbb R$, there exists a unique minimizing geodesic segment contained in $\mathbb C$ joining $z$ and $w$. 

Now assume that $1+z\overline w\in \mathbb R$. It follows that $z,w\neq 0$ and using the action of 
of $K_0$, we may even assume that $z>0$ and $w\in \mathbb R\smallsetminus\{0\}$. For convenience set $z=x$ and $w=y$. The antipodal point of $z$ is equal to $-\frac{1}{x}$. Now if $y\in ]-\frac{1}{x},\infty)$, the geodesic segment from $x$ to $y$ does not contain the antipodal point of $z$ and hence is the minimizing geodesic segment form $z$ to $w$. On the opposite, if $y<-\frac{1}{x}$, the geodesic segment from $z$ to $w$ contains the antipodal point of $z$ and hence is not minimizing.This finishes the proof of the Proposition.
\end{proof}
\begin{proposition} There exists a global smooth determination of $\arg k_c(z,w)$ on the set $\mathcal S$. It can be chosen as $2\Arg\big(k_c(z,w)\big)$, where $\Arg$ is the principal determination of the argument on $\mathbb C\smallsetminus (-\infty, 0]$.
\end{proposition}
Notice that the space $\mathcal S$, seen as $\{(z,w), z\overline w\notin (-\infty,-1]\}$ is star-shaped with respect to $(0,0)$ in $\mathbb C^2$, so that $\mathcal S$ simply connected.
\subsection{Geodesics and triangles in $\mathfrak p_+$}

Inspired by the previous case, we now define a certain subset $\mathcal S$ of $\mathfrak p_+\times \mathfrak p_+$.
\smallskip

\noindent
{\bf Definition.\ } {\sl Let $\mathcal S$ be the set of all $(z,w)\in \mathfrak p_+\times \mathfrak p_+$ for which 
\smallskip

$i)$ there exists a unique minimizing geodesic segment from $\Xi(z)$ to $\Xi(w)$
\smallskip

$ii)$ the minimizing segment is contained in $\im(\Xi)$.} 
\smallskip

In the sequel for $(z,w)\in \mathcal S$, denote by $\beta_{z,w} : [0,1]\longrightarrow  \mathfrak p_+$ the unique curve such that $\Xi\circ \beta_{z,w}$ is the unique minimizing geodesic segment joining $\Xi(z)$ and $\Xi(w)$.

\begin{lemma}\label{Spolysphere}
 Let $(z,w)\in \mathcal S$.   
 Then there exists $u\in U$ such that $u$ is defined along  the curve $\beta_{z,w}$  and $u(z)=0$.
\end{lemma}

\begin{proof} By definition, there exists $\beta : [0,1]\longrightarrow \mathfrak p_+$ such that $z=\beta(0), w=\beta(1)$ and $\Xi\circ \beta$ is a unique minimizing geodesic segment joining $\Xi(z)$ and $\Xi(w)$. By Corollary \ref{geodesicpolysphere}, there exists $u\in U$ such that $u$ maps the geodesic segment into $S(\Gamma)$ and $u(z)=o$. The point $u\big(\Xi(w)\big)$ is the endpoint of a unique minimizing geodesic segment starting at $o$, hence, by Proposition \ref{UMG} does not belong to the cut locus of $o$ and by Proposition \ref{bigcell} does belong to $\im(\Xi)$. In other words, $u$ is defined at $w$. But the same argument applies for $z_t=\beta(t)$ for any $t\in [0,1]$, so that $u$ is defined at any point of the minimizing segment.
\end{proof}

\begin{proposition} Let $(z,w)\in \mathcal S$. Then $ k_c(z,w)\neq 0$.

\end{proposition}

\begin{proof} Let $(z,w)\in \mathcal S$.  By Lemma \ref{Spolysphere} there exists $u\in U$ such that $u(z) =0$ and $u$ is defined at $w$. Now $k_c\big(u(z), u(w)\big) = k_c\big(0,u(w)\big) =1$. Hence by \eqref{covkc} $ k_c\big(z,w)\neq 0$.
\end{proof}

\begin{proposition} The subset $\mathcal S$ is simply connected.
\end{proposition}
\begin{proof} The map $(z,w)\longmapsto (w,w)$ is a deformation retract from $\mathcal S$ onto the diagonal $\diag(\mathfrak p_+)$ in $\mathfrak p_+\times \mathfrak p_+$. In fact  $\beta_{z,w}(t)$ is a continuous family of continuous maps from $\mathcal S$ into $\mathcal S$ which satisfies $\beta_{z,w}(1) =  (w,w)$. As $\diag(\mathfrak p_+)\simeq \mathfrak p_+$ is simply connected, $\mathcal S$ is also simply connected.
\end{proof}
For $(z,w)$ in $\mathcal S$ define $\arg  k(z,w)$ as the unique continuous determination of the argument of $ k(z,w)$ which satisfies
$\arg( k(z,z)=0$. The existence of such an argument is guaranteed by the two last propositions. 
\section{The formula for the symplectic area of 
geodesic triangles}
We are now ready for the proof of the main formula. Recall the formula \eqref{potential} for expressing the K\"ahler form as 
\[ \omega_z = i\partial \overline{\partial} \log k(z,z)\ .
\]
Let $d_\mathbb C = -i(\partial -\overline{\partial})$, and let $\rho$ be the 1-differential form defined by
\[\rho_z = d_\mathbb C \log k(z,z)\ .\]
Then \eqref{potential} is equivalent to
\[ \omega = \frac{1}{2} d\rho
\] 

The next step is to compute $\int_\gamma \rho$ where $\gamma$ is a minizing geodesic segment between two points $z,w$ such that  $(z,w) \in\mathcal S$. We follow an argument due to A. Wienhard (\cite{w}).

\begin{lemma}\label{lem_invariant_cocycle_compact}
  Let $\gamma\colon [a,b]\to\mathfrak p_+$ be a smooth curve segment, and
  suppose that $k_c(\gamma(a),\gamma(b))$ is defined. Assume that
  $g\in U$ is an element such that the action of $g$ is defined on all
  points of $\gamma$, that is, $g\gamma$ is another smooth curve
  segment in $\mathfrak p_+$. Then we have
\begin{equation}
  \label{eq:invariant_cocycle_cpt}
\frac{k_c(g\gamma(a),g\gamma(b))}{k_c(g\gamma(b),g\gamma(a))}\,  \exp i\int\limits_{g\gamma}\rho 
  =\frac{k_c(\gamma(a),\gamma(b))}{k_c(\gamma(b),\gamma(a))}\,\exp  i\int\limits_{\gamma}\rho
\end{equation}
\end{lemma}

\begin{proof}
  Let $z$ be any point on $\gamma$. It follows from
  \eqref{covkc} that 
\begin{equation*}
  k_c\big(g(z),g(z)\big) = j(g,z)k_c(z,z)\overline{j(g,z)}
\end{equation*}
and hence 
\begin{equation*}
  \log k_c\big(g(z),g(z)\big) = \log \vert{j(g,z)}\vert^{2} + \log k_c(z,z).
\end{equation*}
Now from $d_{\mathbb C}\log k(z,z)=\rho_z$ and the above we see that
\begin{align*}
  \int\limits_{g\gamma}\rho &= \int_{\gamma} d_\mathbb C g^*\log k_c(z,z)\\
&= \int\limits_{\gamma} d_\mathbb C \log \vert{j(g,z)}\vert^2
  +\int\limits_{\gamma}d_\mathbb C\log k(z,z),
\end{align*}
and since the action of $g$ is defined along $\gamma$ we may choose a
holomorphic logarithm of $z\mapsto j(g,z)$ along
$\gamma$. This logarithm, denoted $\log j(g,z)$, has real part
$\log\vert{j(g,z)}\vert$ and it follows from the Cauchy-Riemann equations
that $d_{\mathbb C}\log\vert{j(g,z)}\vert = d\Im \log j(g,z)$. Hence, 
$d_{\mathbb C}\log \vert{j(g,z)}\vert^2=2d \Im \log j(g,z)$ and
\begin{equation*}
  \int\limits_{g\gamma}\rho = 2\Im \log j(g,\gamma(b)) - 2\Im \log
  j(g,\gamma(a)) + \int\limits_{\gamma}\rho
\end{equation*}
follows. After taking
exponentials we obtain
\begin{equation}\label{eq:exp_path_integral}
  \exp i\int\limits_{g\gamma}\rho =
  \frac{j(g,\gamma(b))}{\overline{j(g,\gamma(b))}}\frac{\overline{j(g,\gamma(a))}}{j(g,\gamma(a))}
  \exp i\int\limits_{\gamma}\rho.
\end{equation}
Using \eqref{covkc} again, we see that 
\begin{equation}\label{eq:compact_transformation_relative}
 \frac{k_c(g\gamma(a),g\gamma(b))}{k_c(g\gamma(b),g\gamma(a))} =
  \frac{j(g,\gamma(a))}{\overline{j(g,\gamma(a))}}
  \frac{k_c(\gamma(a),\gamma(b))}{k_c(\gamma(b),\gamma(a))} \frac{\overline{j(g,\gamma(b))}}{j(g,\gamma(b))} 
\end{equation}
and upon combining \eqref{eq:exp_path_integral} with
\eqref{eq:compact_transformation_relative} we obtain (\ref{eq:invariant_cocycle_cpt}).
\end{proof}

\begin{lemma}\label{lem_rho_vanishes}
  Let $\gamma\colon [a,b]\to \mathfrak p_+$ be a geodesic segment passing
  through $0$. Then
  $\rho_{\gamma(t)}(\dot\gamma(t))=0$ for all $t$.
\end{lemma}

\begin{proof}
This proof is a variation of the proof of Theorem 9.1 in
\cite{co2} in which a similar statement played a key
role. Since $k_c(z,w)$ is $\Ad (K_0)$-invariant, so if $\rho$ and we
may therefore assume that $\gamma$ runs in $\mathfrak a_+$. Write $\gamma(t) =
\sum_{k=1}^r \gamma_k(t)X_k$ where $\gamma_k :  [a,b]\to \mathbb C$ is a
geodesic in the Riemann sphere $\mathbb C \mathbb P_1$. Put
\[k_{\mathbb C^r}(z,w)=\prod_{k=1}^r (1+z_k\overline{w_k})^2\] 
for
$z=(z_1,\dots,z_r),w=(w_1,\dots,w_r)$ in $\mathbb C^r$. Then
\begin{align*}
  (d_{\mathbb C}\log k_c )(\dot\gamma) &= \frac{p}{2}( d_{\mathbb C} \log k_{\mathbb C^r})
                                (\dot\gamma_1,\dots,\dot\gamma_r) \\
& =\frac{p}{2} \sum_{k=1}^r (d_{\mathbb C}\log k_{\mathbb C} )(\dot\gamma_k)
\end{align*}
and the calculation reduces to the situation $\mathbb C \mathbb P_1$. Here each
$\gamma_k$ is a line through the origin and
$k_{\mathbb C}(z,z)=(1+\vert{z}\vert^2)^2$ and it is straightforward to verify
that $(d_{\mathbb C}\log k_{\mathbb C}(z,z))(\dot\gamma)$ vanishes.
\end{proof}

\begin{theorem}\label{thm:compact_curve_int_rho_argument}
 Let $(z,w)\in \mathcal S$ and let $\gamma=\beta_{z,w}$. Then 
\begin{equation}\label{eq:exponentiated_path_int_compact}
  \exp \frac{1}{i}\int\limits_{\gamma} \rho= \frac{k_c(z,w)}{k_c(w,z)}
\end{equation}
and 
\begin{equation}\label{eq:path_int_arg_compact}
 \frac{1}{2} \int\limits_{\gamma}\rho = -\arg k_c(z,w)\ .
\end{equation}
\end{theorem}

\begin{proof}
  Let $u\in U$ be as in  Lemma \ref{Spolysphere}. As $u(z)=o$ we combine
   Lemma \ref{lem_rho_vanishes} and
  \ref{lem_invariant_cocycle_compact} to obtain
\begin{equation*}
  1 = \frac{k_c(z,w)}{k_c(w,z)}\exp i\int\limits_{\gamma}\rho,
\end{equation*}
  proving \eqref{eq:exponentiated_path_int_compact}. As
  $k_c(w,z)=\overline{k_c(z,w)}$ we have, for all $(z,w)\in \mathcal S$
\begin{equation*}
 \exp\left(-i\int_{\beta_{z,w}} \rho\right)= \exp 2i\arg k_c(z,w),
\end{equation*}
Hence $-\int_{\beta_{z,w}} \rho$ and  $2\arg k_c(z,w)$  are two continuous functions on $\mathcal S$ which differ by a multiple of $2\pi$, are equal to $0$ on the diagonal of $\mathfrak p_+\times \mathfrak p_+$, hence coincide everywhere on $\mathcal S$, so that  \eqref{eq:path_int_arg_compact} holds.
\end{proof} 
Now, if we are given a triple $(z_0,z_1,z_2)$ of points in $\mathfrak p_+$
such that each of the pairs $(z_i,z_j)$ belong to $\mathcal S$, then we may
form an oriented geodesic triangle $\mathcal T = \mathcal T(z_0,z_1,z_2)$ as
follows: The triangle $\mathcal T$ is made up of the three unique
shortest geodesic segments connecting the three vertices $z_0,z_1$,
and $z_2$ with orientation given by traversing the boundary in the
order $z_0 \to z_1\to z_2 \to z_0$. If $\Sigma$ is a smooth surface in
$\mathfrak p^+$ with $\mathcal T=\partial \Sigma$, then $\int_{\Sigma}\omega$ only depends
on the boundary $\mathcal T$ and therefore we will not specify any
particular "filling" of $\mathcal T$.

\begin{theorem}\label{thm:compact_type_area_formula}
  Let $(z_0,z_1,z_2)$ be a triple of points in $\mathfrak p_+$ and suppose
  that each pair $(z_i,z_j)$ belongs to $\mathcal S$. Construct the oriented
  geodesic triangle $\mathcal T = \mathcal T(z_0,z_1,z_2)$ as above. Then
\begin{equation}\label{DT}
  \int\limits_{\Sigma}\omega = - \big(\arg k_c(z_0,z_1)+\arg
  k_c(z_1,z_2)+\arg k_c(z_2,z_0)\big)
\end{equation}
holds for any smooth surface $\Sigma\subset \mathfrak p_+$ with $\mathcal T$ as
its boundary.
\end{theorem}

\begin{proof}
  We have $\omega = \frac{1}{2}d\rho$ so the result follows after an
  application of Stoke's theorem and
  (\ref{eq:path_int_arg_compact}) on each of the three geodesic
  segments of $\mathcal T$.
\end{proof}
In particular, for a geodesic triangle $\mathcal T(0,z_1,z_2)$ with
$(z_1,z_2)\in\mathcal S$ we see that the symplectic area of $\mathcal T$
is given by $-\arg k_c(z_1,z_2)$. This result essentially appears in several
articles by S. Berceanu, see e.g. \cite{b1} and
\cite{b2}. It is proven by direct calculation for the complex
Grassmannian. See also \cite{bs} where the
authors use an embedding of $U/K_0$ into projective space $\mathbb C \mathbb P_N$ and
give an interpretation of the argument of $k_c$ in terms of the
symplectic area of geodesic triangles in the ambient space
$\mathbb C \mathbb P_N$. Hangan and Masala \cite{hm} already proved an
exponentiated version of Theorem \ref{thm:compact_type_area_formula}.

At this point, in order to make connection with the global point of view adopted in Section 4, it is convenient to renormalize the metric and the K\"ahler form so that the metric has maximal holomorphic sectional curvature equal to $+1$. The new metric $\widetilde q$ is equal to $\frac{2}{p}\,q$ (see details in \cite{co2}). Let $\widetilde \omega$ be the corresponding normalized K\"ahler form.

One can show (see e.g. \cite{kor}) that there is a unique $K$-invariant $h(z,w)$ on $\mathfrak p_+\times \mathfrak p_+$, holomorphic in $z$ and conjugate-holomorphic in $w$ such that 
\[h(z,w) = \prod_{j=1}^r (1+z_j\overline w_j)\ ,
\]
when $z=\sum_{k=1}^r z_kX_k$ and $w=\sum_{k=1}^rw_kX_k$ and that 
 $k_c(z,w) = h(z,w)^p$ for any pair of $(z,w)$.  Now set 
 \[\widetilde k_c(z,w) = h(z,w)^2\ .\]  
 
 Recalling  the fact that the factor of normalization from $\omega$ to $\widetilde \omega$ is $\frac{2}{p}$, \eqref{potential} becomes
 \[ \text{for } z\in \mathfrak p_+,\qquad \widetilde {\omega}_z = i\, \partial \overline \partial  \log \widetilde k(z,z)\ .
 \]
 and consequently, the main formula \eqref{DT} remains valid with $\widetilde \omega$ (resp. $\widetilde k_c$) replacing $\omega$ (resp. $k_c$).

\footnotesize{ \noindent Address\\ Mads Aunskj\ae r  Bech 
\quad Institut for Matematiske Fag, Aarhus Universitet, Ny Munkegade 118, 8000 Aarhus C (Denmark)
\\  Jean-Louis Clerc\quad Institut \'Elie Cartan, Universit\'e de Lorraine 54506 Vand\oe uvre-l\`es Nancy (France)\\ Bent \O rsted \quad Institut for Matematiske Fag, Aarhus Universitet, Ny Munkegade 118, 8000 Aarhus C (Denmark)
\smallskip

\noindent \texttt{{madsbech88@gmail.com\\jean-louis.clerc@univ-lorraine.fr \\ orsted@imf.au.dk
}}

\end{document}